\documentclass[11pt,a4paper]{article}
\usepackage{amsmath}
\usepackage{amsthm}
\usepackage{makeidx}
\usepackage{latexsym}
\usepackage{graphicx}
\usepackage{amssymb}
\usepackage{hyperref}
\usepackage{geometry}
\usepackage{stackrel}
\usepackage{bbold}

\newtheorem{theorem}{Theorem}

\newtheorem{proposition}{Proposition}
\newtheorem{corollary}{Corollary}
\newtheorem{remark}{Remark}
\newtheorem{definition}{Definition}
\theoremstyle{remark}
\newtheorem{example}{\textbf{Example}}

\usepackage[latin1]{inputenc}

\DeclareRobustCommand{\stirlingtwo}{\genfrac\{\}{0pt}{}}
\def\s{\atopwithdelims[]}

\title{On the operations of sequences in rings and binomial type sequences}
\author{Stefano Barbero, Umberto Cerruti, Nadir Murru\\ \\
Department of Mathematics G. Peano, University of Turin,\\
Via Carlo Alberto 10, 10123, Torino, ITALY\\ \\
stefano.barbero@unito.it, umberto.cerruti@unito.it, nadir.murru@unito.it}
\date{}

\begin{document}
\maketitle

\begin{abstract}
Given a commutative ring with identity $R$, many different and interesting operations can be defined over the set $H_R$ of sequences of elements in  $R$. These operations can also give $H_R$ the structure of a ring. We study some of these operations, focusing on the binomial convolution product and the operation induced by the composition of exponential generating functions. We provide new relations between these operations and their invertible elements. We also study automorphisms of the Hurwitz series ring, highlighting that some well--known transforms of sequences (such as the Stirling transform) are special cases of these automorphisms. Moreover, we introduce a novel isomorphism between $H_R$ equipped with the componentwise sum and the set of the sequences starting with 1 equipped with the binomial convolution product. Finally, thanks to this isomorphism, we find a new method for characterizing and generating all the binomial type sequences.
\end{abstract}


\section{Introduction}
Given a commutative unitary ring $(R,+,\cdot)$, where as usual the unit, the additive identity and the set of invertible elements will be denoted by 1,0, and $R^{*}$,  
we will  examine  the set, that we call $H_{R}$, of sequences with terms belonging to $R$.
The aim of this paper is to highlight new perspectives in studying the algebraic structures arising when $H_{R}$  is equipped with the most commonly used and interesting operations between sequences.  Specifically, we will provide novel relations between inverse elements with respect to the binomial convolution product (also called Hurwitz product) and with respect to the product defined by the composition of exponential generating functions. Moreover, we will study automorphisms of the Hurwitz series ring and find a new   computable isomorphism between the set of sequences starting with 1 equipped with the Hurwitz product and $H_{R}$ equipped with the componentwise sum. Finally thanks to this isomorphism we will show a straightforward method to generate all the binomial type sequences. 

First of all, let us start with some definitions and notation.

\begin{definition}
 Let us consider the set 
 $$H_{R}=\{(a_n)_{n=0}^{+\infty}=(a_0,a_1,a_2,...): \forall i\geq 0,\quad a_i\in R \}$$ 
of sequences whose terms belongs to $R$. We denote with a bold letter the generic element of $H_{R}$ and we will refer to the $n+1$--th term of $\boldsymbol{a} \in H_{R}$ with the two equivalent notations $\boldsymbol{a}[n]$ or $a_{n}$.

We also define the following two subsets of $H_{R}$ 
$$H^{(n)}_{R}=\{\boldsymbol{a}\in H_{R}: \boldsymbol{a}=(a_0,a_1,\ldots,a_{n-1})\}$$
the set of sequences having lenght $n\geq 0$ and the set
$$H^{m}_{R}=\{\boldsymbol{a} \in H_{R}:\boldsymbol{a}[i]=0, i=0,1,\ldots,m-1\},$$
of sequences having the first $m$ terms equal to zero.We call \emph{zero of order} $m$ any $\boldsymbol{a}\in H^{m}_{R}$.\\
The exponential generating function (e.g.f.) and ordinary generating function (o.g.f.) related to $\boldsymbol{a} \in H_{R}$ will be  
$$ s_{e}(\boldsymbol{a}) = A(t) = \sum_{h=0}^{+\infty} a_h \cfrac{t^h}{h!}, \quad  s_o(\boldsymbol{a}) = \bar A(t) = \sum_{h=0}^{+\infty} a_h t^h,$$ where we denote with $s_{e}$ and $s_{o}$ the bijections that map any sequence $\boldsymbol{a} \in H_{R}$ to its e.g.f. or o.g.f. respectively.
\end{definition}
We define two useful transform acting on sequences $\boldsymbol{a}\in H_{R}$.
\begin{definition}
The transforms $\lambda_-$ (the shift operator) and $\lambda_{+,u}, u\in R$, map any sequence $\boldsymbol{a} \in H_R$ respectively into the sequences
$$\lambda_-(\boldsymbol{a}) := (a_1, a_2, ...), \quad  \lambda_{+,u}(\boldsymbol{a}) := (u, a_0, a_1, ...).$$
\end{definition}
Finally we recall the definition of an important family of partition polynomials.
\begin{definition} \label{def:bell}
Let us consider the sequence of variables $X=(x_1,x_2,\ldots)$. The \emph{complete ordinary Bell polynomials} are defined by
$$B_0(X)=1, \quad \forall n\geq 1 \quad B_n(X)=B_n(x_1,x_2,\ldots,x_n)=\sum_{k=1}^nB_{n,k}(X),$$
where $B_{n,k}(X)$ are the \emph{partial ordinary Bell polynomials},  with
$$ B_{0,0}(X)=1, \quad \forall n\geq 1 \quad B_{n,0}(X)=0,\quad \forall k\geq 1 \quad B_{0,k}(X)=0,$$
$$B_{n,k}(X) =B_{n,k}(x_1,x_2,\ldots,x_{n}) =k!\sum_{\substack {i_1  + 2i_2  +  \cdots  + ni_n = n \\ i_1  + i_2  +  \cdots  + i_n  = k}} \prod_{j=1}^{n}\frac{x_j^{i_j }}{i_{j}!},$$
or, equivalently,
$$B_{n,k}(X)=B_{n,k}(x_1,x_2,\ldots,x_{n-k+1})=k!\sum_{\substack {i_1  + 2i_2  +  \cdots  + (n-k+1)i_{n-k+1} = n-k+1 \\ i_1  + i_2  +  \cdots  + i_{n-k+1}  = k}}\prod_{j=1}^{n-k+1}\frac{x_j^{i_j }}{i_{j}!},$$
satisfying the equality
$$ \left(\sum_{n\geq1}x_nz^n\right)^k=\sum_{n \geq k}B_{n,k}(X)z^n. $$
\end{definition}

The most simple operation we may introduce on $H_R$ is the componentwise sum, i.e., given any $\boldsymbol{a},\boldsymbol{b} \in H_R$, then $\boldsymbol{a} +\boldsymbol{b} =\boldsymbol{c}$, where $c_{n} = a_{n} + b_{n}$, for all $n \geq 0$. In this case, the additive inverse of an element $\boldsymbol{a} \in H_R$  is obviously the sequence $-\boldsymbol{a} = (-a_n)_{n=0}^{+\infty}$ and $(H_R, +)$ is a commutative group, whose identity is $\boldsymbol{0}=(0,0,0,...)$.
On the other hand many different operations, playing the role of a product, can be defined on $H_R$.

\subsection*{The Hadamard product}
Given two sequences $\boldsymbol{a},\boldsymbol{b} \in H_R$, the \emph{Hadamard product} $\boldsymbol{a} \bullet \boldsymbol{b}$ is the componentwise product, i.e., $\boldsymbol{a} \bullet \boldsymbol{b} = \boldsymbol{c}$, where $c_{n}=a_{n}\cdot b_{n}$, for all $n \geq 0$. A sequence $\boldsymbol{a} \in H_R$ is invertible with respect to $\bullet$ if and only if for all $n\geq0$  $a_{n} \in R^{*}$. The identity is clearly the sequence made up with all the elements equal to $1$, which we denote by $\boldsymbol{\underline{1}}$. This product takes its name since the paper of Hadamard \cite{Had}. Some recent and interesting studies on the Hadamard product can be found, e.g., in \cite{All} and \cite{Fill}.

\subsection*{The Hurwitz product}
Another well-known and studied operation is the \emph{Hurwitz product} (also called \emph{binomial convolution product}) that we denote with $\star$. Given $\boldsymbol{a},\boldsymbol{b} \in H_R$, the Hurwitz product is defined as $\boldsymbol{a} \star \boldsymbol{b} =\boldsymbol{c}$, where
$$\forall n\geq0, \quad c_n = \sum_{h = 0}^n \binom{n}{h}a_hb_{n-h}.$$
The identity with respect to the Hurwitz product is the sequence $\boldsymbol{\bar{1}}=(1,0,0,...)$.
Moreover the Hurwitz product and the product between e.g.f.s of sequences are strictly related. Indeed if  $s_{e}(\boldsymbol{a})=A(t)$ and $s_{e}(\boldsymbol{b})=B(t)$  then 
\begin{equation}\label{egf} s_{e}(\boldsymbol{a} \star \boldsymbol{b})=A(t)B(t)=s_{e}(\boldsymbol{a})s_{e}(\boldsymbol{b}).\end{equation}

\begin{remark}
If we consider the set $H_R[[t]]$ of the formal power series of the form $\sum_{h=0}^{+\infty} a_h \cfrac{t^h}{h!}$, then $(H_R[[t]],+,\star)$ is the Hurwitz series ring. The Hurwitz series ring has been extensively studied during the years, latest papers on this subject are, e.g., \cite{Kei}, \cite{Ben}, \cite{Gha}, \cite{Ben2}.
\end{remark}

An element $\boldsymbol{a} \in H_R$ is invertible with respect to $\star$ if and only if  $a_0 \in R^*$. In the following, we will denote by $\boldsymbol{a}^{-1}$ such an inverse and with $H_{R}^{*}$ the set of the invertible sequences in $H_{R}$ with respect to the Hurwitz product $\star$. When it exists, the inverse $\boldsymbol{b}=\boldsymbol{a}^{-1}$ can be recursively evaluated using the relations
$$b_{0}=a_{0}^{-1},\quad b_{n} = -a_{0}^{-1}\sum_{h=1}^n \binom{n}{h} a_{h} b_{n-h}, \quad n\geq1.$$
Furthermore we easily observe from \eqref{egf} that the relation between the e.g.f. of $\boldsymbol{a}=A(t)$ and $\boldsymbol{a}^{-1}$ is  $$s_{e}(\boldsymbol{a}^{-1})=\frac{1}{A(t)}=(s_{e}(\boldsymbol{a}))^{-1}$$
since $s_{e}(\boldsymbol{a}\star \boldsymbol{a}^{-1})=s_{e}(\boldsymbol{\bar{1}})=1.$
We point out that a closed expression for $\boldsymbol{a^{-1}}$ has been found in \cite{bcm} by means of the Bell polynomials. We report this result in the following theorem.

\begin{theorem} \label{thm:hinv}
Let $\boldsymbol{a}\in H_{R}^{*}$, and $\boldsymbol{b}=\boldsymbol{a}^{-1}$. Then for all $n\geq 0$, we have
\begin{equation}\label{hurinv}b_n = \frac{n!B_n(g_0,g_1,g_2,\ldots,g_n)}{a_0},\end{equation}
where
$$\boldsymbol{g}=(g_n)_{n=0}^{+\infty}=\left(-\frac{a_{n+1}}{a_0(n+1)!}\right)_{n=0}^{+\infty}.$$
\end{theorem}

\subsection*{The Cauchy product}
The Cauchy product $*$ between two sequences $\boldsymbol{a},\boldsymbol{b} \in H_R$, defined as $\boldsymbol{a} * \boldsymbol{b} =\boldsymbol{c}$, where $$\forall n \geq 0, \quad c_n = \sum_{h=0}^n a_h b_{n-h},$$ is striclty connected to the Hurwitz product. Indeed, the map
\begin{equation}\label{gamma}\gamma: (H_R,+,*) \rightarrow (H_R,+,\star), \quad \gamma(\boldsymbol{a})[n]=n!\boldsymbol{a}[n], n\geq 0 \end{equation}
is obviously an isomorphism. Thus all the considerations regarding the Hurwitz product $\star$ can be viewed in terms of the Cauchy product, and vice versa. In the following, we will mainly deal with the Hurwitz product $\star$.

\subsection*{Compositions of generating functions}
Finally, we can define two products by composing ordinary and exponential generating functions. Specifically, given $\boldsymbol{a},\boldsymbol{b} \in H_R$, we define $\boldsymbol{a} \circ_o \boldsymbol{b} =\boldsymbol{c}$ and $\boldsymbol{a} \circ_e \boldsymbol{b} =\boldsymbol{d}$, where
$$\boldsymbol{c} = s_o^{-1}\left( \sum_{h=0}^{+\infty} a_h \left( \sum_{k=0}^{+\infty} b_k t^k \right)^h\right), \quad \boldsymbol{d} = s_e^{-1}\left( \sum_{h=0}^{+\infty} \frac{a_h}{h!} \left( \sum_{k=0}^{+\infty} b_k \cfrac{t^k}{k!} \right)^h \right).$$
Let us observe that these products are defined only if the sequence $\boldsymbol{b}$ is a zero of order $\geq 1$.

In the following, we will often use the product $\circ_e$ obtained by the composition of exponential generating functions. Thus, we will write $\circ$ instead of $\circ_e$. It is easy to verify that the sequence $\boldsymbol{1}=\lambda_{+,0}(\boldsymbol{\bar{1}})=(0,1,0,0,...)$ is the identity with respect to the product $\circ$. 

A sequence  $\boldsymbol{a} \in H_R$ is invertible  with respect to the product $\circ$ if and only if $a_0 = 0$ and $a_1 \in R^*$. In this case, the inverse of $\boldsymbol{a}$ that we denote by $\boldsymbol{a}^{(-1)}$ is the sequence in $H_{R}$ whose terms are

\begin{equation}\label{compinv} \boldsymbol{a}^{(-1)}[0]=0, \quad \forall n \geq 1 \quad
\boldsymbol{a}^{(-1)}[n]=\frac{(n-1)!}{a_{1}^{n}}\sum_{j=0}^{n-1}(-1)^{j}{n+j-1 \choose j}B_{n-1,j}\left(\bar{a}_1,\ldots,\bar{a}_{n-j}\right),\end{equation}
where $ \bar{a}_i=\frac{a_{i+1}}{a_1(i+1)!}, i \geq 1$ (see, e.g. \cite{Char}, chapter 11, for a detailed proof).

Given the sequence $\boldsymbol{a}\in H_{R}$, with $a_0 \in R^*$, the terms of the sequence $\boldsymbol{a}^{-1}$ can be explicitly evaluated using (\ref{hurinv}): 
$$\boldsymbol{a}^{-1}=\left(\cfrac{1}{a_0}, -\cfrac{a_1}{a_0^2}, \cfrac{2a_1^2-a_0a_2}{a_0^3},\cfrac{-6a_1^3+6a_0a_1a_2-a_0^2a_3}{a_0^4},...\right).$$
On the other hand, the terms of $\lambda_{+,0}(\boldsymbol{a})^{(-1)}$ can be explicitly evaluated with (\ref{compinv}):
$$\lambda_{+,0}(\boldsymbol{a})^{(-1)}=\left(0,\cfrac{1}{a_0},-\cfrac{a_1}{a_0^3},\cfrac{3a_1^2-a_0a_2}{a_0^5},\cfrac{-15a_1^3+10a_0a_1a_2-a_0^2a_3}{a_0^7},...\right).$$
There is new and simple relation between $\boldsymbol{a}^{-1}$ and $\lambda_{+,0}(\boldsymbol{a})^{(-1)}$, i.e., between the inverse elements with respect to the products $\star$ and $\circ$ respectively, as we will show in the next theorem.
\begin{theorem}
Let $\boldsymbol{a} \in H_R$ be a  $\star$--invertible element, i.e., $a_0 \in R^*$, then
\begin{equation}\label{relinv}\boldsymbol{a}^{-1} = \lambda_{-}\left(\lambda_{+,0}(\boldsymbol{a})^{(-1)}\right) \circ \lambda_{+,0}(\boldsymbol{a}).\end{equation}
\end{theorem}
\begin{proof}
Let  $s_{e}(\boldsymbol{a})=A(t)$, then $$s_{e}(\lambda_+(0,\boldsymbol{a}))=F(t)=\int A(t)dt.$$ If $s_{e}\left(\lambda_{+,0}(\boldsymbol{a})^{(-1)}\right)=G(t)$, by definition of inverse with respect to  $\circ$,  we have
$$F(G(t)) = G(F(t)) = t.$$
Moreover, let us observe that $$s_{e}\left(\lambda_{-}\left(\lambda_{+,0}(\boldsymbol{a})^{(-1)}\right)\right)=G'(t)\quad s_{e}\left(\lambda_{-}\left(\lambda_{+,0}(\boldsymbol{a})^{(-1)}\right) \circ \lambda_{+,0}(\boldsymbol{a})\right)=G'(F(t)).$$ Thus, we have
$$G'(F(t))F'(t) = 1$$
and, since $F'(t) = A(t)$, we obtain
$$G'(F(t)) = \cfrac{1}{A(t)},$$
that completes the proof.
\end{proof}
It seems hard to find a simple relation like \eqref{relinv}  describing $\lambda_{+,0}(\boldsymbol{a})^{(-1)}=\boldsymbol{\lambda}^{(-1)}$ in terms of $\boldsymbol{a}^{-1}$. However, we easily find the following recursive expression

\begin{equation} 
\label{eq:cinv} \forall n\geq0 \quad \boldsymbol{\lambda}^{(-1)}[n+1] = n!\sum_{k=0}^{n} \frac{\boldsymbol{a}^{-1}[k]}{k!}B_{n,k}\left(\bar{\boldsymbol{\lambda}}^{(-1)}[1],\ldots \bar{\boldsymbol{\lambda}}^{(-1)}[n-k+1]\right), 
\end{equation}
where $\bar{\boldsymbol{\lambda}}^{(-1)}=\frac{\boldsymbol{\lambda}^{(-1)}[i]}{i!}$ for all $i\geq 1$. Indeed, using the same notation as in the previous theorem, we know that 
$$F'(G(t))G'(t) = 1, \quad F'(G(t)) = A(G(t)).$$
Thus,
$$G'(t) = \cfrac{1}{A(G(t))}$$
and  since 
$$\cfrac{1}{A(G(t))} = \sum_{k=0}^{+\infty} \frac{\boldsymbol{a}^{-1}[k]}{k!}(G(t))^k = \sum_{n=0}^{+\infty}\left( \sum_{k=0}^n \frac{\boldsymbol{a}^{-1}[k]}{k!}B_{n,k}\left( \bar{\boldsymbol{\lambda}}^{(-1)}[1],\ldots \bar{\boldsymbol{\lambda}}^{(-1)}[n-k+1] \right)\right)t^n,$$ $$G'(t) = \sum_{n=0}^{+\infty}\frac{\boldsymbol{\lambda}^{(-1)}[n+1]}{n!}t^n$$
 we obtain \eqref{eq:cinv}.

\subsection*{Automorphisms}
Some distributive properties hold for the products $\bullet$, $\star$, $\circ$. We summarize them in the following proposition.

\begin{proposition}
 Given any $\boldsymbol{b} \in H_R^0$ and $\boldsymbol{a}, \boldsymbol{c} \in H_R$, we have
\begin{itemize}
\item[-] $(\boldsymbol{a} +\boldsymbol{c}) \circ \boldsymbol{b} = (\boldsymbol{a} \circ \boldsymbol{b}) + (\boldsymbol{c} \circ \boldsymbol{b})$;
\item[-] $(\boldsymbol{a} \star \boldsymbol{c}) \circ \boldsymbol{b} = (\boldsymbol{a} \circ \boldsymbol{b}) \star (\boldsymbol{c} \circ \boldsymbol{b})$.
\end{itemize}
Given the sequence $\boldsymbol{\beta}(r) = (r^{n})_{n=0}^{+\infty}$, for any $r \in R$, and any $\boldsymbol{a},\boldsymbol{c} \in H_R$, we have
\begin{itemize}
\item[-] $(\boldsymbol{a} +\boldsymbol{c}) \bullet \boldsymbol{\beta}(r) = (\boldsymbol{a} \bullet \boldsymbol{\beta}(r)) + (\boldsymbol{c} \bullet \boldsymbol{\beta}(r))$;
\item[-] $(\boldsymbol{a} \star \boldsymbol{c}) \bullet \boldsymbol{\beta}(r) = (\boldsymbol{a} \bullet \boldsymbol{\beta}(r)) \star (\boldsymbol{c} \bullet \boldsymbol{\beta}(r))$.
\end{itemize}
\end{proposition}
\begin{proof}
The proof is straightforward remembering that $+$ and $\star$ represent the sum and the product of formal exponential series, respectively.
\end{proof}

Clearly the product $\bullet$ by $\boldsymbol{\beta}(r)$ may be considered as an endomorphism of the ring $(H_R, +, \star)$. Moreover, we can also introduce a class of endomorphisms  in $(H_R, +, \star)$ by means of the product $\circ$ by  $\boldsymbol{b} \in H_R^0$.  We explicitly state these important and  immediate consequences of the previous proposition
\begin{corollary}
For any $z\in R$  and any fixed sequence $\boldsymbol{b} \in H^{0}_{R}$ the maps
$$\mathcal{H}_{\boldsymbol{\beta}(z)}:(H_R,+,\star) \rightarrow (H_R,+,\star),\quad \mathcal{H}_{\boldsymbol{\beta}(r)}(\boldsymbol{a})=\boldsymbol{a}\bullet \boldsymbol{\beta}(r)$$ 
$$\mu_{\boldsymbol{b}} : (H_R,+,\star) \rightarrow (H_R,+,\star), \quad  \mu_{\boldsymbol{b}}(\boldsymbol{a})=\boldsymbol{a}\circ \boldsymbol{b}$$
 are endomorphisms. Furthermore, if $r,\boldsymbol{b}[1] \in R^{*}$ then $\mathcal{H}_{\boldsymbol{\beta}(z)}$ and $\mu_{\boldsymbol{b}}$ are authomorphisms whose inverses are, respectively, $\mathcal{H}_{\boldsymbol{\beta}(r^{-1})}$ and
$\mu^{-1}_{\boldsymbol{b}}=\mu_{\boldsymbol{b}^{(-1)}}$ and the identity authomorphism is $\mathcal{H}_{\boldsymbol{\underline{1}}}=\mu_{\boldsymbol{1}}$.
\end{corollary}
Some interesting classical transforms acting on sequences may be interpreted as special cases of automorphisms $\mu_{\boldsymbol{b}}$ of $(H_R, +, \star)$ for a suitable choiche of $\boldsymbol{b}$. Let us summarize these transforms in the following definition.

\begin{definition}
Given any sequence $\boldsymbol{a} \in H_R$, we define
\begin{enumerate}
\item the \emph{alternating sign} transform $\mathcal E$ that maps $\boldsymbol{a}$ into  a sequence $\boldsymbol{b} = \mathcal E(\boldsymbol{a}) \in H_R$, whose terms are
$$b_n = (-1)^n a_n;$$
\item the \emph{Stirling transform} $\mathcal S$ that maps $\boldsymbol{a}$ into a sequence $\boldsymbol{b} = \mathcal S(\boldsymbol{a}) \in H_R$, whose terms are
$$b_n = \sum_{h=0}^n \stirlingtwo{n}{h} a_h,$$
where $\stirlingtwo{n}{h}$ are the Stirling numbers of the second kind (see e.g. \cite{Gra}, chapter 6, for definition and properties of Stirling numbers of first and second kind);
\item  the inverse $\mathcal{S}^{-1}$ of the Stirling transform that maps $\boldsymbol{a}$ into a sequence $\boldsymbol{b} = \mathcal S^{-1}(\boldsymbol{a}) \in H_R$, whose terms are
$$b_n = \sum_{h=0}^{n}(-1)^{n-h} {n \brack h} a_h,$$
where $n \s h$ are the (unsigned) Stirling numbers of the first kind.
\end{enumerate}
\end{definition}

The transform $\mathcal{E}$ is often used for studying properties of sequences,
also for its simple but important effect on e.g.f.s, since if  $s_{e}(\boldsymbol{a})=A(t)$, then $s_{e}(\mathcal E(\boldsymbol{a}))=A(-t)$ for all $\boldsymbol{a}\in H_{R}$. Clearly, $\mathcal E$ is an automorphism of $(H_R, +, \star)$, since we have $\mathcal E = \mu_{-\boldsymbol{1}}$.
Furthermore, to clearly identify the sequence $\boldsymbol{b}$ such that $\mathcal{S}=\mu_{\boldsymbol{b}}$ we need the following result 
\begin{proposition} \label{egfstirl}
Given $\boldsymbol{a} \in H_{R} $, if  $s_{e}(\boldsymbol{a})=A(t)$, then $s_{e}(\mathcal{S}(\boldsymbol{a}))= A(e^t-1)$.
\end{proposition}
\begin{proof}
Let us recall that $\stirlingtwo{n}{h} = 0$ when $h > n$ and 
$$\sum_{n=0}^{+\infty} \stirlingtwo{n}{h}\cfrac{t^n}{n!} = \cfrac{(e^t - 1)^h}{h!},$$
see \cite{Gra} pag. 337.
Thus, we have
$$B(t) = \sum_{n=0}^{+\infty}\left( \sum_{h=0}^n \stirlingtwo{n}{h}a_h \right)\cfrac{t^n}{n!} = \sum_{h=0}^{+\infty} a_h \left( \sum_{n=0}^{+\infty} \stirlingtwo{n}{h}\cfrac{t^n}{n!} \right) = \sum_{h=0}^{+\infty} a_h\cfrac{(e^t-1)^h}{h!} = A(e^t-1).$$
\end{proof}
Now, observing that $s_e^{-1}(e^t-1) = (0,1,1,1,...)=\lambda_{+,0}(\boldsymbol{\underline{1}})$, we have
$$\mathcal S = \mu_{\lambda_{+,0}(\boldsymbol{\underline{1})}}.$$
Since from Proposition \ref{egfstirl} we also easily find that $s_{e}(\mathcal {S}^{-1}(\boldsymbol{a}))=A(\log (t+1))$, from the formal power series identity $$\log(t+1)= \sum_{h=1}^{+\infty} (-1)^{h-1}(h-1)! \cfrac{t^h}{h!},$$ we have 
$$\mathcal S^{-1} = \mu_{\lambda_{+,0}(\mathcal{E}(\boldsymbol{f}))},$$
where $\boldsymbol{f}= (0!, 1!, 2!, ...)$ is the sequence of the factorial numbers A000142 in Oeis \cite{Oeis}.

\section{Isomorphisms}

In the following we focus on a subgroup of $H_R^{*}$, the set of invertible elements of $H_R$ with respect to the Huwitz product. Specifically, we consider
$$U_R = \{\boldsymbol{a} \in H_R^*: a_0 = 1 \},$$ and we indicate by $U_{R}^{(n)}$ the set of sequences in $U_{R}$ of lenght $n$. Furthermore we will deal with products and powers of sequences with respect to the Hurwitz product.

Given $\boldsymbol{a} \in U_R$ with $s_{e}(\boldsymbol{a})=A(t)$, we let $\boldsymbol{a}^x$ denote the sequence such that $s_{e}(\boldsymbol{a}^x)=(A(t))^x$. The sequence $\boldsymbol{a}^x$ can be viewed as a sequence of polynomials in the variable $x$ and we use the notation $p^{(\boldsymbol{a})}(x) = \boldsymbol{a}^x$ in order to highlight this, i.e., $$p^{(\boldsymbol{a})}(x) = \left(p_n^{(\boldsymbol{a})}(x)\right)_{n=0}^{+\infty} = (\boldsymbol{a}^x[n])_{n=0}^{+\infty},$$ where $p_n^{(\boldsymbol{a})}(x) \in R[x]$. Clearly  we also have 
\begin{equation} \label{px+y}
p^{(\boldsymbol{a})}(x+y)=p^{(\boldsymbol{a})}(x)\star p^{(\boldsymbol{a})}(y) \end{equation}
since $(A(t))^{x+y}=(A(t))^{x}(A(t))^{y}$.
In the next proposition, we will provide a closed form for the coefficients of the polynomials $p_n^{(\boldsymbol{a})}(x)$.

\begin{proposition} \label{prop:pan}
Given the polynomial $p_n^{(\boldsymbol{a})}(x) = \sum_{j=0}^n c_j x^j$, then
\begin{equation}\label{cj} c_j = \sum_{h=0}^{n-j} (-1)^h \cfrac{n!}{(h+j)!}  {h + j \brack j} B_{n,h+j}(a_1,\ldots,a_{n-h-j+1}),\end{equation}
for $j =0,1,2,...,n$, where ${k \brack j}$ are the (unsigned) Stirling numbers of the first kind.
\end{proposition}
\begin{proof}
Since $a = (1, a_1, a_2, ...)$ has e.g.f. $A(t) = 1 + \sum_{n=1}^{+\infty} \cfrac{a_n}{n!}t^n$, we have
\begin{align*}
(A(t))^x& =\left(1 + \sum_{n=1}^{+\infty} \frac{a_n}{n!}t^n\right)^x=\sum_{k=0}^{+\infty}{x\choose k} \sum_{n\geq k} B_{n,k}(a_1,\ldots,a_{n-k+1})t^n =\\
&= \sum_{n=0}^{+\infty}\cfrac{1}{n!}\left( n! \sum_{k=0}^n \binom{x}{k}B_{n,k}(a_1,\ldots,a_{n-k+1}) \right)t^n.
\end{align*}
Thus,
\begin{equation}\label{pnx} p_n^{(\boldsymbol{a})}(x) = n!\sum_{k=0}^n  {x\choose k} B_{n,k}(a_1,\ldots,a_{n-k+1}).\end{equation}
Moreover, from equation 6.13 in \cite{Gra}, we can write
\begin{equation}\label{bingen}
\binom{x}{k} = \cfrac{(-1)^k}{k!} \sum_{j=0}^k (-1)^j {k \brack j} x^j\end{equation}
and finally, substituting this equality in (\ref{pnx}) 
$$p_n^{(\boldsymbol{a})}(x) = \sum_{j=0}^n \left( n! \sum_{k=j}^n B_{n,k}(a_{1},\ldots,a_{n-k+1}) \cfrac{(-1)^k}{k!}(-1)^j{k \brack j} \right) x^j$$
from which, replacing the index $k$ with $h=k-j$, the thesis follows.
\end{proof}

Interesting consequences arise studying the sequences  $\boldsymbol{a}^x \in U_R$ with respect to the Hurwitz product.
From (\ref{cj}) or evaluating the sequence of formal derivatives of $(A(t))^x$ 
$$x( A(t))^{x-1} A'(t), x (A(t))^{x-1} A''(t) + (x-1) x (A(t))^{x-2} A'(t)^2, ...$$
we easily obtain that $p_0^{(\boldsymbol{a})}(x) = 1$ and $x \vert p_n^{(\boldsymbol{a})}(x)$ for $n\geq 1$. Now if we consider, for example, $R = \mathbb Z_n$, we have for any $\boldsymbol{a} \in U_{\mathbb Z_n}$,  $\boldsymbol{a}^n = (1,0,0,...) =\boldsymbol{ \bar{1}}$. Moreover, let $p$ be a prime number, then any element different from $\boldsymbol{\bar{1}}$ in $U_{\mathbb Z_p}$ has period $p$. Thus $\left(U^{(n)}_{\mathbb Z_p}, \star\right)$ is a finite group of order $p^{n-1}$ and by the structure theorem for finite abelian groups, we have the following isomorphism 
\begin{equation}\label{isop}\left(U^{(n)}_{\mathbb Z_p},\star\right) \simeq \left(H^{(n-1)}_ {\mathbb Z_p},+\right).\end{equation}
Since (\ref{isop}) holds for every prime $p$ we also have for all $n\geq 2$
$$\left(U^{(n)}_{\mathbb Z},\star\right) \simeq \left(H^{(n-1)}_{\mathbb Z},+\right)$$
and consequently
\begin{equation}\label{isoz}\left(U_{\mathbb Z},\star\right) \simeq \left(H_{ \mathbb Z},+\right).\end{equation}
The interesting isomorphism (\ref{isoz}) is only a special case of a more general computable isomorphism that we can provide between $\left(U_{R},\star \right)$ and $\left(H_{R},+ \right)$.
Let us consider the infinite set of sequences $\mathcal B = \{\boldsymbol{b}^{(1)}, \boldsymbol{b}^{(2)},\boldsymbol{b}^{(3)},...\}$, where
$$\boldsymbol{\underline{1}}=\boldsymbol{b}^{(1)}=(1,1,1,1,...), \boldsymbol{b}^{(2)}=(1,0,1,0,...),\boldsymbol{b}^{(3)}=(1,0,0,1,0...),\ldots $$
i.e., for $i \geq 2$, $\boldsymbol{b}^{(i)}[i]=1$ and $\boldsymbol{b}^{(i)}[k]=0$ for any $k \not = 0, i$.
\begin{definition}
We define two elements 
$\boldsymbol{a}$ and $\boldsymbol{b}$ of $U_{R}$ as \emph{independent} if $<\boldsymbol{a}> \cap <\boldsymbol{b}> =\boldsymbol{\bar{1}}$, where $<\boldsymbol{u}>$ is the subgroup of $(U_R,\star)$ generated by $\boldsymbol{u} \in U_R$. 
\end{definition}
 Clearly  the sequences $\boldsymbol{b}^{(i)}$ are mutually independent, since $$p^{\left(\boldsymbol{b}^{(1)}\right)}(x)=\left(\boldsymbol{b}^{(1)}\right)^{x} =(x^n)_{n=0}^{+\infty}$$ and considering $p^{\left(\boldsymbol{b}^{(i)}\right)}(x)=\left(\boldsymbol{b}^{(i)}\right)^{x}$, we surely have $$\forall i\geq 2,\quad p^{\left(\boldsymbol{b}^{(i)}\right)}_{1}(x)=\cdots=p^{\left(\boldsymbol{b}^{(i)}\right)}_{i-1}(x)=0.$$ 
Thus we may consider the set $\mathcal B$ as a basis of $(U_{R},\star)$.

\begin{remark}
The sequences $p^{\left(\boldsymbol{b}^{(i)}\right)}(x)$ can be evaluated in a fast way by the following formulas. We have for $i\geq 2$
\begin{equation} \label{eq:bpot}p^{\left(\boldsymbol{b}^{(i)}\right)}_{0}(x) = 1, \quad p^{\left(\boldsymbol{b}^{(i)}\right)}_{n}(x) = \begin{cases} 0, \quad n \not\equiv 0 \pmod i \cr {x\choose k}\frac{(ik)!}{(i!)^k}=\cfrac{(ik)!}{k!(i!)^k} \prod_{h=1}^{k}(x-h+1), \quad n=ki  \end{cases}. \end{equation}
and when $i=1$ obviously
\begin{equation}\label{bpot1} \forall n \geq 0 \quad p^{\left(\boldsymbol{b}^{(1)}\right)}_n(x)=x^n\end{equation}
The equalities \eqref{eq:bpot} and \eqref{bpot1} easily follow by relation (\ref{pnx}) in Proposition \ref{prop:pan} and Definition \ref{def:bell} of partial ordinary Bell polynomials, or considering for all $i\geq 1$ the e.g.f.s $s_{e}\left(\left(\boldsymbol{b}^{(i)}\right)^x\right)$. 
\end{remark}

Now we can explicitly define an isomorphism between $\left(H^{(n-1)}_{R}, +\right)$ and $\left(U^{(n)}_{R}, \star\right)$ by means of the basis $\mathcal B$. 

\begin{theorem}
Let $\tau^{(n)}: \left(H^{(n-1)}_{R}, + \right) \rightarrow \left(U^{(n)}_{R}, \star\right)$ be defined as follows
$$\tau^{(n)}: \boldsymbol{x}=(x_0,...,x_{n-2}) \mapsto \prod_{k=1}^{n-1} \left(\boldsymbol{b}^{(k,n)}\right)^{x_{k-1}}=\prod_{k=1}^{n-1} p^{\left(\boldsymbol{b}^{(k,n)}\right)}(x_{k-1})$$
where $\boldsymbol{b}^{(k,n)}$ is the sequence $\boldsymbol{b}^{(k)}$ of length $n$, $n \geq 2$. Then $\tau^{(n)}$ is an isomorphism.
\end{theorem}
\begin{proof}
Given any $\boldsymbol{x},\boldsymbol{y} \in H^{(n-1)}_{R}$ as a straightforward consequence of (\ref{px+y}) we have
$$\tau^{(n)}(\boldsymbol{x}+\boldsymbol{y}) = \tau^{(n)}(\boldsymbol{x}) \star \tau^{(n)}(\boldsymbol{y}).$$
Moreover, given any $\boldsymbol{a} = (1,a_1,...,a_{n-1}) \in U_{R}^{(n)}$ there exists one and only one sequence $\boldsymbol{x} \in H^{(n-1)}_{R}$ such that $\tau^{(n)}(\boldsymbol{x})=\boldsymbol{a}$.
To prove this claim, first of all we describe explicitly the terms of $\tau^{(n)}(\boldsymbol{x})$, using equations  \eqref{eq:bpot}, \eqref{bpot1}, and the definition of Hurwitz product $\star$

\begin{equation}\label{taui}\begin{cases}
\tau^{(n)}(\boldsymbol{x})[0]=1,\\
\tau^{(n)}(\boldsymbol{x})[1]=x_0,\\
\tau^{(n)}(\boldsymbol{x})[m]=m!\underset{\sum_{l=1}^{m}lj_{l}=m}{\sum}\frac{x^{j_1}_{0}}{j_{1}!}\stackrel[h=2]{m}{\prod}\frac{{x_{h-1}\choose j_{h}}}{(h!)^{j_h}},\quad m=2,\ldots,n-1\end{cases}\end{equation}
then, if we consider the equation $\tau^{(n)}(\boldsymbol{x})=\boldsymbol{a}$, we can solve it determining in a unique way every element of $\boldsymbol{x}$. Indeed, from \eqref{taui} is straightforward to observe that for  $i=1,\ldots,n-1$ the term $x_{i-1}$ appears for the first time and with power 1 in the expression of
$\tau^{(n)}(\boldsymbol{x})[i]$. Thus we obtain from \eqref{taui} the following system of equations
\begin{equation}\label{sis} x_0=a_1,\quad x_1+x^{2}_{0}=a_2, \quad x_{i-1} + P(x_0,...,x_{i-2}) =\tau^{(n)}(\boldsymbol{x})[i]= a_i,\quad i=3,\ldots,n-1, \end{equation}
where $$P(x_0,...,x_{i-2})=i!\sum_{\sum_{l=1}^{i-1}lj_{l}=i}\frac{x^{j_1}_{0}}{j_{1}!}\prod_{h=2}^{i-1}\frac{{x_{h-1}\choose j_{h}}}{(h!)^{j_h}}$$ 
and clearly the unique solution to \eqref{sis} can be recursively computed.
\end{proof}
\begin{example}
Given $\boldsymbol{x}= (x_0,x_1,x_2,x_3) \in H^{(4)}_{R}$, we have from \eqref{taui}
$$\tau^{(5)}(\boldsymbol{x}) = (1, x_0, x_0^2+x_1, x_0^3+3x_0x_1+x_2, x_0^4-3x_1+6x_0^2x_1+3x_1^2+4x_0x_2+x_3.)$$
Conversely knowing $\boldsymbol{a}=(1,a_1,a_2,a_3,a_4) \in U^{(5)}_{R}$, recursively solving the corresponding equations \eqref{sis}, we get the unique solution 
$\boldsymbol{x}\in H^{(4)}_{R}$  to $\tau^{(5)}(\boldsymbol{x})=\boldsymbol{a}$ given by
 $$x_0=a_1,\quad x_1=a_2-a_{1}^2$$
 $$x_2=a_3-P(x_0,x_1)=a_3-P(a_1,a_2-a_{1}^2)=a_3+2a_1^3-3a_1a_2$$
 \begin{align*}x_3&=a_4-P(x_0,x_1,x_2)=a_4-P(a_1,a_2-a_{1}^2,a_3+2a_1^3-3a_1a_2)=\\
 &=a_4-3a_1^2-6a_1^4+3a_2+12a_1^2a_2-3a_2^2-4a_1a_3.\end{align*}
\end{example}

As an immediate consequence of the previous theorem, we can define and evaluate the isomorphism $$\tau: \left(H_{R},+\right) \rightarrow \left(U_{R},\star \right)$$  acting as  $\tau^{(n+1)}$ on the first $n$  elements of a sequence $\boldsymbol{x}\in H_{R}$, for all $n\geq1$.

Since $R$ is a generic unitary commutative ring, we point out that all the results we have proved also hold if we consider the ring of polynomials $R[x]$, instead of $R$, dealing with sequences whose terms are polynomials.
In the next section, we will consider binomial type sequences, which are strictly connected with the polynomials $p^{(\boldsymbol{a})}(x)$ and we will use the isomorphism 
$$\tau: \left(H_{R[x]},+\right) \rightarrow \left(U_{R[x]},\star \right)$$
in order to give a method for the construction of all the binomial type sequences.

\section{Binomial type sequences}

A sequence of polynomials $q(x) = (q_n(x))_{n=0}^{+\infty}$, $q_n(x) \in R[x]$, is called a \emph{binomial type sequence} if $q_0(x) = 1$ and
$$q_n(x+y) = \sum_{h=0}^n \binom{n}{h}q_n(x)q_{n-h}(x).$$
In terms of the Hurwitz product $\star$  the above condition can be restated as
$$q(x+y) = q(x) \star q(y),$$
and clearly  $q(x)\in U_{R[x]}$.

The binomial type sequences are a very important subject deeply  studied during the years. They were introduced in \cite{Rota} and used in the theory of umbral calculus. On the other hand, binomial type sequences are widely used in combinatorics and probability. For instance, in \cite{Sch} the author showed the use of binomial type sequences to solve many ''tiling'' problems and in \cite{Mih} the author established new relations between binomial type sequences and Bell polynomials. Further interesting results can be found in \cite{Kisil}, \cite{Goss} and \cite{Bucc}. Moreover, many well--known and useful polynomials are binomial type sequences, like Abel, Laguerre and Touchard polynomials.

It is easy to see, as we have already pointed out in \eqref{px+y}, that the polynomial sequences of the kind $p^{(\boldsymbol{a})}(x)$ are binomial type sequences. Moreover, we have the following result, i.e., any binomial type sequences is of the form $p^{(\boldsymbol{a})}(x)$ for a certain $\boldsymbol{a}\in U_{R}$.

\begin{proposition}
If $q(x)$ is a binomial type sequence, then there exists $ \boldsymbol{a} \in U_R$ such that $q(x) = p^{(\boldsymbol{a})}(x)$.
\end{proposition}
\begin{proof}
It suffices to prove that $s_{e}(q(x))=(A(t))^x$, where $A(t)=s_{e}(\boldsymbol{a})$ for a suitable sequence $\boldsymbol{a} \in U_R$. Let $c_i$ be the coefficient of $x$ in the polynomial $q_i(x)$. Then  $s_{e}(q(x))=\exp(xf(t))$, where $f(t) = \sum_{n=1}^\infty c_n \cfrac{t^n}{n!}$,  as proved in \cite{Rota}. Thus  we have only to choose $\boldsymbol{a}=s_{e}^{-1}(\exp(f(t))$.
\end{proof}
The result showed in the previous proposition has been also stated in \cite{Bucc2}. In the next theorem we highlight a novel characterization of the binomial type sequences by means of the isomorphism $\tau$ and the basis $\mathcal B$ defined in the previous section.

\begin{theorem}
For any $\boldsymbol{u} = (u_0, u_1, ...)\in H_R$, let $x\boldsymbol{u} = (xu_0, xu_1, ....)$ be a sequence of polynomials in $R[x]$, if we consider the isomorphism
$$\tau: \left(H_{R[x]},+\right) \rightarrow \left(U_{R[x]},\star \right)$$ 
 then $\tau(x\boldsymbol{u})$ is a binomial type sequence. Conversely, any binomial type sequence is of the form $\tau(x\boldsymbol{u})$ for a certain sequence $\boldsymbol{u} \in H_R$.
\end{theorem}
\begin{proof}
The sequence $\tau(x\boldsymbol{u})$ satisfies the relation
$$\tau((x+y)\boldsymbol{u})=\tau(x\boldsymbol{u}+y\boldsymbol{u})=
\tau(x\boldsymbol{u})\star\tau(y\boldsymbol{u})$$
and $\tau(x\boldsymbol{u})[0]=1$, thus $\tau(x\boldsymbol{u})$ is a binomial type sequence.
On the other hand, since $\tau$ is an isomorphism, given a binomial type sequence  $q(x) \in U_{R[x]}$ there exists a unique sequence $r(x) = (r_i(x))_{i=0}^{+\infty} \in H_{R[x]}$ such that
$$\tau^{-1}(q(x))=r(x).$$ 
Moreover we must have 
\begin{equation}\label{rrrr} r(x+y)=r(x)+r(y)\end{equation}
indeed
$$r(x+y)=\tau^{-1}(q(x+y)) = \tau^{-1}(q(x)\star q(y))=\tau^{-1}(q(x))+\tau^{-1}(q(y))=r(x)+r(y),$$
The equation \eqref{rrrr} is satisfied in $H_{R[x]}$ if and only if
$$\forall i \geq 0 \quad r_i(x) = xu_i, \quad u_i \in R.$$ Hence, $q(x) = \tau(x\boldsymbol{u})$ for a certain sequence $\boldsymbol{u}\in H_{R}$.
\end{proof}
We would like to point out that the previous theorem give a computable method to evaluate binomial type sequences, starting from any sequence in $H_R$.
For example, if we  consider the Fibonacci sequence $\boldsymbol{F}=(1,1,2,3,5,8,13,...)$ then, using (\ref{taui}), we obtain the binomial type sequence
$$\tau(x\boldsymbol{F}) = (1, x, x^2+x, x^3+3x^2+2x, x^4+6x^3+11x^2, x^5 +10x^4 +35x^3+20x^2 +5x,\ldots).$$
 On the other hand, given any binomial type sequence in $U_{R[x]}$, we have a computable method to find the sequence in $H_R$ that generates it.
In fact we provide in the following theorem a recursive formula which gives all the terms of the sequence $\boldsymbol{u}\in H_{R}$  such that $\tau(x\boldsymbol{u})=q(x)$ for a given binomial type sequence $q(x)$.
\begin{theorem}
Let us consider a binomial type sequence $q(x)=\left(q_{n}(x)\right)_{n=0}^{+\infty}$, where $q_{0}(x)=1$ and $q_{n}(x)=\sum_{i=1}^{n}c_{i,n}x^{i}$. Then we have $\tau(x\boldsymbol{u})=q(x)$ if  
\begin{equation}\label{umx}
\forall m\geq 1, \quad u_{m-1}=c_{1,m}+\underset{
k|m,
k\neq1,m
}{\sum}\frac{\left(-1\right)^{\frac{m}{k}}m!}{\frac{m}{k}\left(k!\right)^{\frac{m}{k}}}u_{k-1}.
\end{equation}
\end{theorem}
\begin{proof}
From the definition of $\tau$ and \eqref{taui} we have 
$$\tau^{\left(n\right)}\left(x\boldsymbol{u}\right)\left[0\right]=1=q_{0}\left(x\right), \quad
\tau^{\left(n\right)}\left(x\boldsymbol{u}\right)\left[1\right]=xu_{0}=c_{1,1}x=
q_{1}\left(x\right)$$
and for all $m\geq 2$, the polynomial $q_{m}(x)$ is equal to
\begin{equation}\label{tumx}
\tau^{\left(n\right)}\left(x\boldsymbol{u}\right)\left[m\right]=m!\underset{\stackrel[l=1]{m}{\sum}lj_{l}=m}{\sum}\frac{\left(c_{1,1}x\right)^{j_{1}}}{j_{1}!}\stackrel[h=2]{m}{\prod}\frac{\binom{
u_{h-1}x}
{j_{h}}
}{\left(h!\right)^{j_{h}}}.
\end{equation}
Clearly $u_0=c_{1,1}$ and, in order to find the value of $u_{m-1}$ for $m\geq 2$, we only need to find the coefficient of $x$ in \eqref{tumx}, since it must be equal to $c_{1,m}$. We observe that in the summation showed in \eqref{tumx} we find $x$ multiplied by a suitable coefficient only when we consider solutions of $\sum_{l=1}^{m}lj_l=m$ having the form $j_k=\frac{m}{k}$,   with $j_l=0$ for all $l\neq k$, where $k|m$ and $k\neq 1$. Indeed, when $k=m$ we obtain the summand $u_{m-1}x$ and, for all $k|m$, $k\neq 1,m$, we easily find from 
$\binom{
u_{k-1}x}
{j_{k}}
\frac{m!}{\left(k!\right)^{j_{k}}}$
the summands 
$$
\frac{\left(-1\right)^{j_{k}+1}}{\left(k!\right)^{j_{k}}}\frac{m!}{j_{k}!}{
j_{k}
\brack
1}
\left(u_{k-1}x\right)=\frac{\left(-1\right)^{j_{k}+1}m!}{j_{k}\left(k!\right)^{j_{k}}}u_{k-1}x
$$
using \eqref{bingen}, since ${j_k \brack 1}=\left(j_{k}-1\right)!.$ 
Thus we have 
$$
u_{m-1}+\underset{
k|m,
k\neq1,m
}{\sum}\frac{\left(-1\right)^{\frac{m}{k}+1}m!}{\frac{m}{k}\left(k!\right)^{\frac{m}{k}}}u_{k-1}=c_{1,m},
$$ 
from which equation \eqref{umx} easily follows.
\end{proof}
\begin{remark}
	As a straightforward consequence we observe that if $p$ is a prime number then $u_{p-1}=c_{1,p}$.
\end{remark}

It would be interesting for further research to focus on the case $R = \mathbb Z$, in order to find new relations between integer sequences and polynomial sequences of binomial type. We provide some examples listed in the table below, where, for some well--known binomial type sequences, we give the first terms of the corresponding sequence  arising from (\ref{umx}).

\small
\begin{center}
\begin{tabular}{|c|c|c|}
	\hline 
Binomial type polynomials	& Expression & Corresponding $\boldsymbol{u}$ \\ 
	\hline 

Increasing powers	& $x^{n}$  & $\boldsymbol{\bar{1}} = (1,0,0,0,...)$ \\ 
	\hline
Laguerre	& $L_n(x)=\sum_{k=0}^n \frac{n!}{k!}\binom{n-1}{k-1}(-x)^k$  & (-1,-2,-6,-30,-120,-720,-5040,...) \\ 
	\hline 
Touchard	& $T_n(x) = \sum_{k=0}^{n}\stirlingtwo{n}{k}x^k$ & (1,1,1,4,1,-19,1,771,-559,...) \\ 
	\hline 
Abel	& $A_n(x,a)=x(x-an)^{n-1}$  & $(1, -2 a, 9 a^2, -6a -64 a^3, 625 a^4,\ldots)$ \\ 
	\hline 
Pochhammer	& $(x)_n=x(x+1)(x+2)\cdots(x+n-1)$  & $(1, 1, 2 , 9 , 24 , 110 , 720 , 5985 , 39200 \ldots)$ \\ 

	\hline 
\end{tabular} 
\end{center}
\normalsize

As far as we know, none of the sequences $\boldsymbol{u}$ different from $\boldsymbol{\bar{1}}$, seems to be already recorded in Oeis \cite{Oeis}.

\end{document}